\documentclass[11pt]{amsart}

\usepackage[a4paper,margin=1in]{geometry}
\usepackage{amsmath,amssymb,amsthm,mathtools}
\usepackage{hyperref}

\newtheorem{theorem}{Theorem}
\newtheorem{lemma}{Lemma}
\newtheorem{remark}{Remark}

\newcommand{\D}{\mathbb D}
\newcommand{\C}{\mathbb C}
\newcommand{\Log}{\operatorname{Log}}
\newcommand{\Arg}{\operatorname{Arg}}

\title[A bounded quasiconformal harmonic map]
{A bounded globally univalent quasiconformal harmonic map whose analytic part is unbounded}

\author{David Kalaj}

\address{University of Montenegro, Faculty of Natural Sciences and Mathematics,
Cetinjski put b.b., 81000 Podgorica, Montenegro}

\email{davidk@ucg.ac.me}



\begin{document}
\begin{abstract}
We construct, for every \(0<k<1\), a bounded globally univalent harmonic
mapping
\[
        f=h+\overline g \colon \D\to\C
\]
such that
\[
        |g'(z)|\le k|h'(z)|,\qquad z\in\D,
\]
while the analytic part \(h\) is unbounded. The construction is based on a
bounded logarithmic spiral map on a far right horizontal strip, together
with a smaller logarithmic perturbation.
\end{abstract}
\maketitle
\section{Introduction}
The class of quasiconformal harmonic mappings has been studied extensively
from the point of view of boundary regularity, Lipschitz and co-Lipschitz
estimates, boundary correspondence, and distortion. A fundamental theorem
of Pavlovi\'c asserts that a harmonic homeomorphism of the unit disk onto
itself is quasiconformal if and only if it is bi-Lipschitz; this may also be
formulated in terms of the bi-Lipschitz character of the boundary function
and the boundedness of the Hilbert transform of its derivative
\cite{Pavlovic2002}. Kalaj and Pavlovi\'c obtained an analogous boundary
correspondence theorem for quasiconformal harmonic diffeomorphisms of a
half-plane \cite{KalajPavlovic2005}. Kalaj subsequently proved several regularity results for quasiconformal
harmonic mappings between smooth Jordan domains. In particular, he proved
Lipschitz estimates for mappings between \(C^{1,\mu}\)-smooth Jordan
domains, and bi-Lipschitz estimates under the additional assumption that
the image domain is convex \cite{KalajMathZ2008}. Kalaj proved several regularity results for quasiconformal harmonic
mappings between smooth Jordan domains. In particular, he proved Lipschitz
estimates for mappings between \(C^{1,\mu}\)-smooth Jordan domains, and
bi-Lipschitz estimates under the additional assumption that the image
domain is convex \cite{KalajMathZ2008}. Related bi-Lipschitz results for
smoother domains were obtained in \cite{KalajPisa2011,KalajMateljevic2012}.
The \(C^{1,\alpha}\) Lyapunov-domain case was subsequently treated by
Bo\v zin and Mateljevi\'c, who proved co-Lipschitz estimates for harmonic
quasiconformal mappings onto Lyapunov Jordan domains \cite{BozinMateljevicPisa2020}.
Later, Kalaj showed that if the domains have only \(C^1\) boundaries, then
a quasiconformal harmonic mapping is globally H\"older continuous for every
exponent \(0<\alpha<1\), although it need not be Lipschitz in general
\cite{Kalaj2022RMI}.

Other results of Kalaj connect this theory with harmonic measure, weighted
estimates, and elliptic equations. In \cite{Kalaj2015AdvMath}, he extended
Lavrentiev's \(A_\infty\)-equivalence between harmonic measure and
arc-length measure in chord-arc domains, and used this to obtain a
Lindel\"of-type theorem for quasiconformal harmonic mappings. Kalaj and
Pavlovi\'c also studied quasiconformal self-mappings of the disk satisfying
the Poisson equation
\[
        \Delta w=g,
\]
showing that the corresponding family is uniformly Lipschitz, and, under a
smallness assumption on \(g\), uniformly bi-Lipschitz
\cite{KalajPavlovic2011TAMS}. In a different direction, Kalaj obtained
Riesz-type inequalities for harmonic mappings in the unit disk
\cite{Kalaj2019TAMS}. Further related work includes results of Kalaj and
Mateljevi\'c on \((K,K')\)-quasiconformal harmonic mappings
\cite{KalajMateljevic2012}, Partyka and Sakan on quasiconformal and
Lipschitz harmonic mappings onto bounded convex domains
\cite{PartykaSakan2014}, Partyka--Sakan--Zhu on mappings with convex
holomorphic part \cite{PartykaSakanZhu2018}, and Chen and collaborators on
John disks, boundary behavior, and biharmonic variants
\cite{ChenPonnusamy2017,ChenWang2020}.

The present paper concerns a different aspect of the theory. We construct a
bounded, globally univalent quasiconformal harmonic mapping
\[
        f=h+\overline g
\]
in the unit disk whose analytic part \(h\) is nevertheless unbounded. Thus,
although quasiconformal harmonic mappings enjoy strong regularity under
many geometric and analytic assumptions, boundedness of the full harmonic
mapping does not force boundedness of the analytic part in its canonical
decomposition.
\section{Statement of the result}

\begin{theorem}\label{th:1}
Let \(0<k<1\). There exists a bounded globally one-to-one harmonic mapping
\[
        f=h+\overline g \colon \D\to\C
\]
such that
\[
        |g'(z)|\le k|h'(z)|,\qquad z\in\D,
\]
but \(h\) is unbounded in \(\D\).
\end{theorem}

Recall that for a sense-preserving harmonic mapping
\[
        f=h+\overline g,
\]
the analytic dilatation is
\[
        \omega(z)=\frac{g'(z)}{h'(z)}.
\]
Thus the condition
\[
        |\omega(z)|\le k<1
\]
is equivalent to
\[
        |g'(z)|\le k|h'(z)|.
\]
It implies that \(f\) is quasiconformal with maximal dilatation
\[
        K=\frac{1+k}{1-k}.
\]

\section{A stable logarithmic spiral lemma}
\begin{lemma}[Stable logarithmic spiral lemma]
Let
\[
0<\alpha<\frac12,\qquad b>0,
\]
and, for \(A>0\), let
\[
R_A=\{s=x+iy:x>A,\ |y|<b\}.
\]
Define
\[
\Psi(s)=s^{-\alpha}e^{-is^\alpha},
\]
using the principal branch, and set
\[
Q(s)=\Arg\Log s
\]
where defined.

Then, for \(A\) sufficiently large, \(Q\) is a smooth real-valued function on \(R_A\), the map \(\Psi\) is one-to-one on \(R_A\), and
\[
\eta(A):=
\sup_{\substack{s,t\in R_A\\s\neq t}}
\frac{|Q(s)-Q(t)|}{|\Psi(s)-\Psi(t)|}
\longrightarrow 0
\qquad\text{as } A\to+\infty.
\]
Consequently, for every fixed \(\varepsilon>0\), if \(A\) is sufficiently large, then
\[
T_\varepsilon(s)=\Psi(s)-2\varepsilon Q(s)
\]
is one-to-one on \(R_A\).
\end{lemma}

\begin{proof}
We first enlarge \(A\), if necessary, so that \(A>2\). Then \(R_A\) lies in the right half-plane, the principal logarithm is defined there, and
\[
\Re \Log s=\log |s|>0
\]
for \(s\in R_A\). Hence \(\Log s\) lies in the right half-plane, so \(Q(s)=\Arg\Log s\) is a smooth real-valued function on \(R_A\).

All asymptotic estimates below are uniform for \(|y|<b\) as \(x\to+\infty\).
Write
\[
        \Psi(s)=e^{-V(s)-iU(s)},
\]
where
\[
        U(s)=\operatorname{Re}(s^\alpha)+\alpha\Arg s,
        \qquad
        V(s)=\alpha\log|s|-\operatorname{Im}(s^\alpha).
\]
Let \(s=x+iy\), with \(|y|<b\). All estimates below are uniform for
\(|y|<b\) as \(x\to+\infty\). We have
\[
        U(s)=x^\alpha+O(x^{-1}),
        \qquad
        V(s)=\alpha\log x-\alpha yx^{\alpha-1}+O(x^{-2}).
\]
Put
\[
        u=U(s),
        \qquad
        \beta=\frac{1-\alpha}{\alpha}.
\]
Then
\[
        x\asymp u^{1/\alpha},
        \qquad
        V(s)=\log u+O(u^{-\beta}).
\]
Since \(0<\alpha<1/2\), we have \(\beta>1\). Hence
\[
        e^{-V(s)}
        =
        u^{-1}\left(1+O(u^{-\beta})\right),
\]
and therefore
\[
        \Psi(s)
        =
        u^{-1}e^{-iu}+O(u^{-1-\beta})
        =
        u^{-1}e^{-iu}+O(u^{-1/\alpha}).
\]
If
\[
        \gamma(u)=u^{-1}e^{-iu},
\]
then
\[
        \Psi(s)=\gamma(U(s))+O(U(s)^{-1/\alpha}).
\]

We shall use the following elementary observation. After increasing \(A\)
if necessary,
\[
        U_x>0
        \qquad
        \text{on }
        \{x+iy:x>A/2,\ |y|<b\}.
\]
Indeed,
\[
        U_x
        =
        \alpha\operatorname{Re}(s^{\alpha-1})
        -
        \alpha\frac{y}{x^2+y^2}
        =
        \alpha x^{\alpha-1}(1+o(1)).
\]
Thus \((u,y)=(U(x+iy),y)\) gives global coordinates on this strip.

Moreover, if \(s,t\in R_A\), then the coordinate rectangle in the
\((u,y)\)-variables joining \(s\) to \(t\) is contained in
\[
        \{x+iy:x>A/2,\ |y|<b\}.
\]
Indeed, from
\[
        U(x+iy)=x^\alpha+O(x^{-1})
\]
there is \(C>0\) such that
\[
        x^\alpha-Cx^{-1}
        \le U(x+iy)
        \le x^\alpha+Cx^{-1}
\]
for \(x\) large. If \(u_*\) lies between \(U(s)\) and \(U(t)\), and
\(y_*\) lies between \(\operatorname{Im}s\) and \(\operatorname{Im}t\), then
\[
        u_*\ge A^\alpha-CA^{-1},
\]
whereas
\[
        U(A/2+iy_*)\le (A/2)^\alpha+CA^{-1}.
\]
For \(A\) sufficiently large,
\[
        A^\alpha-CA^{-1}>(A/2)^\alpha+CA^{-1}.
\]
Thus
\[
        u_*>U(A/2+iy_*).
\]
Since \(U_x>0\) for \(x>A/2\), the point \(x_*\) determined by
\[
        U(x_*+iy_*)=u_*
\]
satisfies \(x_*>A/2\). This proves the claim.

We first prove that \(\Psi\) is one-to-one on \(R_A\). Suppose
\[
        \Psi(s)=\Psi(t).
\]
Write
\[
        u=U(s),\quad v=V(s),
        \qquad
        u'=U(t),\quad v'=V(t).
\]
Since
\[
        \Psi(s)=e^{-v-iu},
        \qquad
        \Psi(t)=e^{-v'-iu'},
\]
we get
\[
        v=v',
        \qquad
        u-u'\in 2\pi\mathbb Z.
\]

First assume \(u=u'\). Then \(s=t\). Indeed, along the level curve
\(U(x+iy)=u\), write \(x=X(u,y)\). Since
\[
        U_x\frac{dx}{dy}+U_y=0,
\]
we have
\[
        \frac{dx}{dy}=-\frac{U_y}{U_x}.
\]
Therefore
\[
\begin{aligned}
        \frac{d}{dy}V(X(u,y)+iy)
        &=
        V_x\frac{dx}{dy}+V_y        \\
        &=
        V_y-\frac{V_xU_y}{U_x}.
\end{aligned}
\]
Using
\[
        U_x=\alpha x^{\alpha-1}+O(x^{\alpha-2}),
        \qquad
        U_y=\frac{\alpha}{x}+O(x^{\alpha-2}),
\]
and
\[
        V_x=\frac{\alpha}{x}+O(x^{\alpha-2}),
        \qquad
        V_y=-\alpha x^{\alpha-1}+o(x^{\alpha-1}),
\]
we get
\[
        \frac{d}{dy}V(X(u,y)+iy)
        =
        -\alpha x^{\alpha-1}+o(x^{\alpha-1})<0
\]
for \(A\) sufficiently large. Hence \(V\) is strictly monotone on each
level curve of \(U\). Thus
\[
        U(s)=U(t),
        \qquad
        V(s)=V(t)
\]
imply \(s=t\).

Now suppose, after interchanging \(s\) and \(t\) if necessary, that
\[
        u'=u+2\pi N,
        \qquad
        N\in\mathbb N.
\]
Using
\[
        V(s)=\log u+O(u^{-\beta}),
        \qquad
        V(t)=\log u'+O((u')^{-\beta}),
\]
we obtain
\[
        |V(t)-V(s)|
        \ge
        |\log u'-\log u|-O(u^{-\beta})-O((u')^{-\beta}).
\]
If \(u\le u'\le 2u\), then
\[
        |\log u'-\log u|
        =
        \log\left(1+\frac{u'-u}{u}\right)
        \ge
        \frac12\frac{u'-u}{u}
        \ge
        \frac{\pi}{u}.
\]
If \(u'>2u\), then
\[
        |\log u'-\log u|\ge \log 2.
\]
In either case, since \(\beta>1\), the error terms cannot cancel the
logarithmic difference for \(A\) large. Hence \(V(t)\ne V(s)\), a
contradiction. Therefore \(N=0\), and so \(s=t\). Thus \(\Psi\) is
one-to-one on \(R_A\).

It remains to prove \(\eta(A)\to0\). First,
\[
        \Log s=\log|s|+i\Arg s.
\]
Since \(|\Arg s|\le C/x\) on \(R_A\),
\[
        Q(s)
        =
        \Arg\Log s
        =
        \frac{\Arg s}{\log|s|}
        +
        O\left(
        \frac{|\Arg s|^3}{(\log|s|)^3}
        \right).
\]
Consequently,
\[
        |Q(s)|\le \frac{C}{x\log x}
        \le
        \frac{C}{u^{1/\alpha}\log u}.
\]
Also straightforward calculations yields,
\[
        Q_x=O\left(\frac{1}{x^2\log x}\right),
        \qquad
        Q_y=O\left(\frac{1}{x\log x}\right).
\]

In the global coordinates \((u,y)=(U(x+iy),y)\), we have
\[
        x_u=\frac1{U_x}=O(x^{1-\alpha}),
        \qquad
        x_y=-\frac{U_y}{U_x}=O(x^{-\alpha}).
\]
Therefore
\[
        \left|\frac{\partial Q}{\partial u}\right|
        \le
        \frac{C}{u^2\log u},
        \qquad
        \left|\frac{\partial Q}{\partial y}\right|
        \le
        \frac{C}{u^{1/\alpha}\log u}.
\]
Similarly,
\[
        \left|\frac{\partial V}{\partial u}\right|
        \le
        \frac{C}{u},
        \qquad
        \frac{\partial V}{\partial y}
        =
        -\alpha x^{\alpha-1}+o(x^{\alpha-1}).
\]
Since \(x\asymp u^{1/\alpha}\), this gives
\[
        c u^{-\beta}
        \le
        \left|\frac{\partial V}{\partial y}\right|
        \le
        C u^{-\beta}.
\]

Now take \(s,t\in R_A\), and write
\[
        u=U(s),\quad v=V(s),
        \qquad
        u'=U(t),\quad v'=V(t).
\]
Assume without loss of generality that \(u\le u'\). We split into three
cases.

\medskip

\noindent
\textbf{Case 1: \(|u-u'|\le\pi\).}
Since \(s,t\in R_A\), we have \(u,u'\gtrsim A^\alpha\).  After increasing
\(A\), we may assume \(u,u'>2\pi\).  Hence, if \(|u-u'|\le\pi\), then
\[
        \frac12 u\le u'\le 2u,
\]
and therefore \(u\asymp u'\).
Also
\[
        |v-v'|=O(u^{-1}),
\]
because
\[
        v=\log u+O(u^{-\beta}),
        \qquad
        v'=\log u'+O((u')^{-\beta}).
\]
Since
\[
        \Psi(t)-\Psi(s)
        =
        e^{-v-iu}
        \left(e^{-(v'-v)-i(u'-u)}-1\right),
\]
and since \(|u'-u|\le\pi\), \(|v'-v|\le1\) for \(A\) large, we have
\[
        |\Psi(s)-\Psi(t)|
        \ge
        c u^{-1}\left(|u-u'|+|v-v'|\right).
\]

Let \(y=\operatorname{Im}s\) and \(y'=\operatorname{Im}t\). Integrating in
the \((u,y)\)-coordinate rectangle gives
\[
        |Q(s)-Q(t)|
        \le
        \frac{C}{u^2\log u}|u-u'|
        +
        \frac{C}{u^{1/\alpha}\log u}|y-y'|.
\]
Moreover,
\[
        |V(u',y)-V(u,y)|
        \le
        \frac{C}{u}|u-u'|,
\]
while
\[
        |V(u',y')-V(u',y)|
        \ge
        c u^{-\beta}|y-y'|.
\]
Hence
\[
        u^{-\beta}|y-y'|
        \le
        C\left(|v-v'|+u^{-1}|u-u'|\right).
\]
Since \(1/\alpha=1+\beta\), we obtain
\[
        |Q(s)-Q(t)|
        \le
        \frac{C}{u\log u}|v-v'|
        +
        \frac{C}{u^2\log u}|u-u'|.
\]
Therefore
\[
        |Q(s)-Q(t)|
        \le
        \frac{C}{\log u}\,
        u^{-1}\left(|u-u'|+|v-v'|\right)
        \le
        \frac{C}{\log u}|\Psi(s)-\Psi(t)|.
\]
This ratio tends to \(0\) uniformly as \(A\to+\infty\).

\medskip

\noindent
\textbf{Case 2: \(u'\ge 2u\).}

Since
\[
        |\Psi(s)|=u^{-1}(1+o(1)),
        \qquad
        |\Psi(t)|=(u')^{-1}(1+o(1)),
\]
we get
\[
        |\Psi(s)-\Psi(t)|
        \ge
        \bigl||\Psi(s)|-|\Psi(t)|\bigr|
        \ge
        c u^{-1}.
\]
On the other hand,
\[
        |Q(s)-Q(t)|
        \le
        \frac{C}{u^{1/\alpha}\log u}.
\]
Thus
\[
        \frac{|Q(s)-Q(t)|}{|\Psi(s)-\Psi(t)|}
        \le
        C\frac{u^{1-1/\alpha}}{\log u}
        \to0.
\]

\medskip

\noindent
\textbf{Case 3: \(|u-u'|>\pi\) and \(u'<2u\).}

Since we have assumed \(u\le u'\), the additional condition \(u'<2u\)
gives
\[
        u\le u'<2u.
\]
Hence \(u\asymp u'\). For
\[
        \gamma(u)=u^{-1}e^{-iu},
\]
we have
\[
        |\gamma(u)-\gamma(u')|
        \ge
        \left|\frac1u-\frac1{u'}\right|
        =
        \frac{|u'-u|}{uu'}
        \ge
        c u^{-2}.
\]
Using
\[
        \Psi(s)=\gamma(u)+O(u^{-1/\alpha}),
        \qquad
        \Psi(t)=\gamma(u')+O(u^{-1/\alpha}),
\]
and since \(1/\alpha>2\), we get
\[
        |\Psi(s)-\Psi(t)|
        \ge
        c u^{-2}
\]
for \(A\) sufficiently large. Meanwhile,
\[
        |Q(s)-Q(t)|
        \le
        \frac{C}{u^{1/\alpha}\log u}.
\]
Therefore
\[
        \frac{|Q(s)-Q(t)|}{|\Psi(s)-\Psi(t)|}
        \le
        C\frac{u^{2-1/\alpha}}{\log u}
        \to0,
\]
because \(2-1/\alpha<0\).

Combining the three cases gives
\[
        \eta(A)\to0.
\]

Finally, suppose
\[
        T_\varepsilon(s)=T_\varepsilon(t).
\]
Then
\[
        \Psi(s)-\Psi(t)
        =
        2\varepsilon\bigl(Q(s)-Q(t)\bigr).
\]
Hence
\[
        |\Psi(s)-\Psi(t)|
        \le
        2\varepsilon\eta(A)|\Psi(s)-\Psi(t)|.
\]
Choose \(A\) so large that \(2\varepsilon\eta(A)<1\). Then
\[
        \Psi(s)=\Psi(t).
\]
Since \(\Psi\) is one-to-one on \(R_A\), we get \(s=t\). Thus
\(T_\varepsilon\) is one-to-one on \(R_A\).
\end{proof}
\section{Proof of the main result}
\begin{proof}[Proof of Theorem~\ref{th:1}]
Fix \(0<k<1\). Choose
\[
        0<\alpha<\frac12
\]
and fix \(\varepsilon>0\). We shall choose a sufficiently large number
\(A>0\) below.

Define
\[
        S(z)=A-\Log(1-z),\qquad z\in\D,
\]
where the principal logarithm is used. Since \(1-z\) lies in the right
half-plane for \(z\in\D\), the function \(\Log(1-z)\) is single-valued and
analytic in \(\D\). Moreover,
\[
        |\Arg(1-z)|<\frac{\pi}{2},
\]
and \(|1-z|<2\). Hence
\[
        S(\D)\subset
        \{s=x+iy:x>A-\log 2,\ |y|<\pi/2\}.
\]
Put
\[
        A_0=A-\log 2.
\]
By increasing \(A\), we may assume that \(A_0\) is as large as needed in all
subsequent arguments.

Let
\[
        \Psi(s)=s^{-\alpha}e^{-is^\alpha},
\]
where the principal branch is used. We apply the stable logarithmic spiral
lemma with
\[
        b=\frac{\pi}{2}
\]
and with the lower real-part bound \(A_0\). Thus, after increasing \(A\) if
necessary, the map
\[
        T_\varepsilon(s)=\Psi(s)-2\varepsilon \Arg\Log s
\]
is one-to-one on the horizontal strip containing \(S(\D)\).

We now define analytic functions \(h\) and \(g\) in \(\D\) by
\[
        h(z)=\Psi(S(z))+i\varepsilon\Log\Log S(z),
\]
and
\[
        g(z)=i\varepsilon\Log\Log S(z).
\]
For \(A\) sufficiently large, \(S(\D)\) lies in a far right half-plane, and
\(\Log S(z)\) also lies in the right half-plane. Hence \(\Log\Log S(z)\) is
well-defined and analytic in \(\D\).

Set
\[
        f(z)=h(z)+\overline{g(z)}.
\]
Then \(f\) is harmonic in \(\D\). Moreover,
\[
\begin{aligned}
        f(z)
        &=
        \Psi(S(z))
        +i\varepsilon\Log\Log S(z)
        -i\varepsilon\overline{\Log\Log S(z)}  \\
        &=
        \Psi(S(z))-2\varepsilon \Arg\Log S(z).
\end{aligned}
\]
Therefore
\[
        f=T_\varepsilon\circ S.
\]
Since \(S\) is conformal in \(\D\) and \(T_\varepsilon\) is one-to-one on
the strip containing \(S(\D)\), it follows that \(f\) is globally
one-to-one in \(\D\).

Next we prove that \(f\) is bounded. Let \(s=x+iy\), with
\[
        x>A_0,\qquad |y|<\frac{\pi}{2}.
\]
On this strip,
\[
        \Im(s^\alpha)=O(x^{\alpha-1}),
        \qquad x\to+\infty,
\]
uniformly in \(y\). Hence
\[
        |\Psi(s)|
        =
        |s|^{-\alpha}e^{\Im(s^\alpha)}
\]
is bounded on \(S(\D)\). Also, since \(S(\D)\) lies in a far right
horizontal strip, the quantity
\[
        \Arg\Log S(z)
\]
is bounded. Therefore
\[
        f(z)=\Psi(S(z))-2\varepsilon\Arg\Log S(z)
\]
is bounded in \(\D\).

We now show that the analytic part \(h\) is unbounded. Take \(z=r\in(0,1)\)
and let \(r\to1^{-}\). Then
\[
        S(r)=A-\log(1-r)\to+\infty.
\]
Consequently,
\[
        \Log\Log S(r)\to+\infty,
\]
whereas
\[
        \Psi(S(r))
        =
        S(r)^{-\alpha}e^{-iS(r)^\alpha}
        \to0.
\]
Thus
\[
        h(r)
        =
        \Psi(S(r))+i\varepsilon\Log\Log S(r)
        \to i\infty.
\]
Hence \(h\) is unbounded in \(\D\).

It remains to prove the quasiconformal estimate
\[
        |g'(z)|\le k |h'(z)|,\qquad z\in\D.
\]
We have
\[
        S'(z)=\frac1{1-z}.
\]
Also,
\[
        \Psi(s)=s^{-\alpha}e^{-is^\alpha},
\]
and therefore
\[
\begin{aligned}
        \Psi'(s)
        &=
        -\alpha s^{-\alpha-1}e^{-is^\alpha}
        -i\alpha s^{-1}e^{-is^\alpha}  \\
        &=
        -\frac{i\alpha e^{-is^\alpha}}{s}
        \left(1-i s^{-\alpha}\right).
\end{aligned}
\]
Hence
\[
        g'(z)
        =
        \frac{i\varepsilon}{S(z)\Log S(z)}S'(z),
\]
and
\[
        h'(z)
        =
        \left(
        \Psi'(S(z))
        +
        \frac{i\varepsilon}{S(z)\Log S(z)}
        \right)S'(z).
\]
Thus
\[
        \omega(z):=\frac{g'(z)}{h'(z)}
        =
        \frac{i\varepsilon}
        {S(z)\Log S(z)\Psi'(S(z))+i\varepsilon},
\]
whenever the denominator is nonzero.

Using the formula for \(\Psi'\), we obtain
\[
        S(z)\Log S(z)\Psi'(S(z))
        =
        -i\alpha\Log S(z)e^{-iS(z)^\alpha}
        \left(1-iS(z)^{-\alpha}\right).
\]
On the far right horizontal strip containing \(S(\D)\), the factor
\[
        |e^{-iS(z)^\alpha}|
\]
is bounded above and below by positive constants, uniformly in \(z\). Also,
\[
        \left|1-iS(z)^{-\alpha}\right|\to1
\]
uniformly as \(A\to+\infty\). Hence there exists a constant \(c>0\), independent
of \(A\), such that for all sufficiently large \(A\),
\[
        \left|
        S(z)\Log S(z)\Psi'(S(z))
        \right|
        \ge
        c\log A_0,
        \qquad z\in\D.
\]
Therefore,
\[
\begin{aligned}
        \left|
        S(z)\Log S(z)\Psi'(S(z))+i\varepsilon
        \right|
        &\ge
        \left|
        S(z)\Log S(z)\Psi'(S(z))
        \right|-\varepsilon  \\
        &\ge
        c\log A_0-\varepsilon .
\end{aligned}
\]
Choose \(A\) so large that
\[
        c\log A_0\ge \varepsilon\left(1+\frac1k\right).
\]
Then
\[
        c\log A_0-\varepsilon\ge \frac{\varepsilon}{k},
\]
and hence
\[
        |\omega(z)|
        \le
        \frac{\varepsilon}{c\log A_0-\varepsilon}
        \le k,
        \qquad z\in\D.
\]
In particular, \(h'(z)\neq0\) in \(\D\), and
\[
        |g'(z)|\le k |h'(z)|,
        \qquad z\in\D.
\]
Consequently,
\[
        J_f=|h'|^2-|g'|^2\ge (1-k^2)|h'|^2>0.
\]
Thus \(f\) is sense-preserving and locally diffeomorphic. Since \(f\) is
globally one-to-one, invariance of domain implies that \(f\) is a
homeomorphism of \(\D\) onto the bounded domain \(f(\D)\). Its maximal
dilatation is at most
\[
        K=\frac{1+k}{1-k}.
\]
Together with the boundedness of \(f\) and the unboundedness of \(h\)
proved above, this completes the proof.
\end{proof}

\begin{remark}
This phenomenon is complementary to the regularity theory for
quasiconformal harmonic mappings. Results of Pavlovi\'c, Kalaj,
Mateljevi\'c, Partyka, Sakan, Zhu, Chen and others give strong Lipschitz,
co-Lipschitz, H\"older, boundary-correspondence, and weighted-measure
conclusions under suitable assumptions on the mapping, the image domain, the
boundary correspondence, or the governing elliptic equation. The present
construction shows that such regularity properties do not imply boundedness
of the individual analytic summands in the canonical representation
\[
        f=h+\overline g .
\]
\end{remark}

\section*{Acknowledgements}

The author would like to thank Anton Baranov for posing this problem and
for helpful discussions. The author gratefully acknowledges financial
support from the grants ``Mathematical Analysis, Optimisation and Machine
Learning'' and ``Complex--analytic and geometric techniques for
non-Euclidean machine learning: theory and applications''.


\begin{thebibliography}{99}

\bibitem{AstalaIwaniecMartin}
K. Astala, T. Iwaniec and G. Martin,
\emph{Elliptic Partial Differential Equations and Quasiconformal Mappings
in the Plane},
Princeton Mathematical Series, Vol. 48,
Princeton University Press, Princeton, 2009.

\bibitem{BozinMateljevicPisa2020}
V. Bo\v zin and M. Mateljevi\'c,
\emph{Quasiconformal and HQC mappings between Lyapunov Jordan domains},
Ann. Sc. Norm. Super. Pisa Cl. Sci. (5) \textbf{21} (2020),
Special Issue, no. 1--2, 107--132.

\bibitem{ChenPonnusamy2017}
S. Chen and S. Ponnusamy,
\emph{John disks and \(K\)-quasiconformal harmonic mappings},
J. Geom. Anal. \textbf{27} (2017), no. 2, 1468--1488.

\bibitem{ChenWang2020}
S. Chen and X. Wang,
\emph{Bi-Lipschitz characteristic of quasiconformal self-mappings of the
unit disk satisfying the bi-harmonic equation},
Indiana Univ. Math. J. \textbf{70} (2021), no. 3, 1055--1086.

\bibitem{ClunieSheilSmall}
J. Clunie and T. Sheil-Small,
\emph{Harmonic univalent functions},
Ann. Acad. Sci. Fenn. Ser. A I Math. \textbf{9} (1984), 3--25.

\bibitem{Duren}
P. Duren,
\emph{Harmonic Mappings in the Plane},
Cambridge Tracts in Mathematics, Vol. 156,
Cambridge University Press, Cambridge, 2004.

\bibitem{KalajMathZ2008}
D. Kalaj,
\emph{Quasiconformal and harmonic mappings between Jordan domains},
Math. Z. \textbf{260} (2008), 237--252.

\bibitem{Kalaj2009}
D. Kalaj,
\emph{Lipschitz spaces and harmonic mappings},
Ann. Acad. Sci. Fenn. Math. \textbf{34} (2009), no. 2, 475--485.

\bibitem{KalajPisa2011}
D. Kalaj,
\emph{Harmonic mappings and distance function},
Ann. Sc. Norm. Super. Pisa Cl. Sci. (5) \textbf{10} (2011),
no. 3, 669--681.

\bibitem{Kalaj2015AdvMath}
D. Kalaj,
\emph{Muckenhoupt weights and Lindel\"of theorem for harmonic mappings},
Adv. Math. \textbf{280} (2015), 301--321.

\bibitem{Kalaj2019TAMS}
D. Kalaj,
\emph{On Riesz type inequalities for harmonic mappings on the unit disk},
Trans. Amer. Math. Soc. \textbf{372} (2019), no. 6, 4031--4051.

\bibitem{Kalaj2022RMI}
D. Kalaj,
\emph{Harmonic quasiconformal mappings between \(C^1\) smooth Jordan domains},
Rev. Mat. Iberoam. \textbf{38} (2022), no. 1, 95--111.

\bibitem{KalajMateljevic2012}
D. Kalaj and M. Mateljevi\'c,
\emph{\((K,K')\)-quasiconformal harmonic mappings},
Potential Anal. \textbf{36} (2012), no. 1, 117--135.

\bibitem{KalajPavlovic2005}
D. Kalaj and M. Pavlovi\'c,
\emph{Boundary correspondence under quasiconformal harmonic diffeomorphisms
of a half-plane},
Ann. Acad. Sci. Fenn. Math. \textbf{30} (2005), no. 1, 159--165.

\bibitem{KalajPavlovic2011TAMS}
D. Kalaj and M. Pavlovi\'c,
\emph{On quasiconformal self-mappings of the unit disk satisfying Poisson's
equation},
Trans. Amer. Math. Soc. \textbf{363} (2011), no. 8, 4043--4061.

\bibitem{LehtoVirtanen}
O. Lehto and K. I. Virtanen,
\emph{Quasiconformal Mappings in the Plane},
Grundlehren der mathematischen Wissenschaften, Vol. 126,
Springer, Berlin, 1973.

\bibitem{PartykaSakan2014}
D. Partyka and K.-i. Sakan,
\emph{Quasiconformal and Lipschitz harmonic mappings of the unit disk onto
bounded convex domains},
Ann. Acad. Sci. Fenn. Math. \textbf{39} (2014), no. 2, 811--830.

\bibitem{PartykaSakanZhu2018}
D. Partyka, K.-i. Sakan and J.-F. Zhu,
\emph{Quasiconformal harmonic mappings with the convex holomorphic part},
Ann. Acad. Sci. Fenn. Math. \textbf{43} (2018), no. 1, 401--418.

\bibitem{Pavlovic2002}
M. Pavlovi\'c,
\emph{Boundary correspondence under harmonic quasiconformal homeomorphisms
of the unit disk},
Ann. Acad. Sci. Fenn. Math. \textbf{27} (2002), no. 2, 365--372.

\end{thebibliography}
\end{document}